\documentclass{amsart}

\usepackage{mathtools}
\usepackage{amsmath}
\usepackage{amssymb}
\usepackage{mathrsfs}
\usepackage{amsthm}
\usepackage{bbold}
\usepackage{esint}
\usepackage{todonotes}
\usepackage{enumerate}
\usepackage{bm}

\newtheorem{Theorem}{Theorem}

\newtheorem{Lemma}[Theorem]{Lemma}
\newtheorem{Proposition}[Theorem]{Proposition}

\theoremstyle{definition}

\newtheorem{Remark}[Theorem]{Remark}
\numberwithin{Theorem}{section}                 

\newcommand{\norm}[1]{\left\|{#1}\right\|}

\renewcommand{\bar}{\overline}
\renewcommand{\epsilon}{\varepsilon}
\newcommand{\T}{\mathbb T}

\newcommand{\N}{\mathbb N}
\renewcommand{\P}{\mathbb P}
\newcommand{\E}{\mathbb E}

\newcommand{\R}{\mathbb R}

\newcommand{\1}{\mathbb 1}

\newcommand{\dual}[2]{\left\langle#1,#2\right\rangle}

\let \div \relax

\DeclareMathOperator{\div}{div}

\title[A dichotomy concerning Riesz transforms]{A dichotomy concerning uniform boundedness of Riesz transforms on Riemannian Manifolds}
\author{Alex Amenta}
\address{Delft Institute of Applied Mathematics \\ Delft University of Technology \\ P.O. Box 5031\\ 2600 GA Delft \\The Netherlands}
\email{amenta@fastmail.fm}
\author{Leonardo Tolomeo}
\address{School of Mathematics\\ The University of Edinburgh and Maxwell Institute for Mathematical Sciences\\ James Clerk Maxwell Building\\ Rm 5210 
The King's Buildings\\ Peter Guthrie Tait Road\\ 
Edinburgh\\ EH9 3FD\\ United Kingdom}
\email{L.Tolomeo@sms.ed.ac.uk}
\thanks{The first author was supported by the VIDI subsidy 639.032.427 of the Netherlands Organisation for Scientific Research (NWO). The second author was supported by the European Research Council (grant no. 637995 ``ProbDynDispEq") and by The Maxwell Institute Graduate School in Analysis and its Applications, a Centre for Doctoral Training funded by the UK Engineering and Physical Sciences Research Council (grant EP/L016508/01), the Scottish Funding Council, Heriot-Watt University and the University of Edinburgh.}

\begin{document}

\begin{abstract}
  Given a sequence of complete Riemannian manifolds $(M_n)$ of the same dimension, we construct a complete Riemannian manifold $M$ such that for all $p \in (1,\infty)$ the $L^p$-norm of the Riesz transform on $M$ dominates the $L^p$-norm of the Riesz transform on $M_n$ for all $n$.
  Thus we establish the following dichotomy: given $p$ and $d$, either there is a uniform $L^p$ bound on the Riesz transform over all complete $d$-dimensional Riemannian manifolds, or there exists a complete Riemannian manifold with Riesz transform unbounded on $L^p$.
\end{abstract}

\maketitle

\section{Introduction}

Given a Riemannian manifold $M$, one can consider the Riesz transform $R := \nabla (-\Delta)^{\frac12}$, where $\nabla$ is the Riemannian gradient and $\Delta$ is the (negative) Laplace--Beltrami operator.
In the Euclidean case $M = \R^n$, this can be identified with the vector of classical Riesz transforms $(R_1,\ldots,R_n)$, as can be seen by writing $R$ as a Fourier multiplier (see \cite[\textsection 5.1.4]{lG14}).

It is easy to show that $R$ is bounded from $L^2(M)$ to $L^2(M;TM)$, and substantially harder to determine whether $R$ extends to a bounded map from $L^p(M)$ to $L^p(M;TM)$ for $p \neq 2$.
We let
\begin{equation*}
  R_p(M) := \sup_{\norm{f}_{L^p} \le 1} \norm{R(f)}_{L^p}
\end{equation*}
denote the (possibly infinite) $L^p$-norm of the Riesz transform on $M$.
Various conditions, often involving the heat kernel on $M$ and its gradient, are known to imply finiteness of $R_p(M)$; see for example \cite{AC05,ACDH04,BF16,gC17,CCH06,CCFR17,CD99,CD03,hL99,LZ17}.
These results usually entail finiteness of $R_p(M)$ for all $p \in (1,2)$, or for some range of $p > 2$.
On the other hand, there exist manifolds $M$ for which $R_p(M)$ is known to be infinite for some (or all) $p > 2$: see \cite{aA17, gC17, CCH06, CCFR17, CD99, hL99}.

\begin{Remark}
  When $M$ has finite volume we abuse notation and write $L^p(M)$ to denote the space of $p$-integrable functions \emph{with mean zero}.
  This modification ensures that $(-\Delta)^{-1/2}$ is densely defined.
  When $M$ has infinite volume, $L^p(M)$ denotes the usual Lebesgue space.
\end{Remark} 

The Euclidean case is now classical: for all $p \in (1,\infty)$ there is a constant $C_p < \infty$ such that $R_p(\R^n) \leq C_p < \infty$ for all $n \in \N$ (\cite{eS83}).
This behaviour is expected to persist for all complete Riemannian manifolds, at least for $p < 2$.
More precisely, in \cite{CD03} it is conjectured that for all $p \in (1,2)$ there exists a constant $C_p < \infty$ such that $R_p(M) \leq C_p$ for all complete Riemannian manifolds $M$.
Such uniform bounds have been proven for all $p \in (1,\infty)$ under curvature assumptions; rather than provide an overview of the vast literature on this topic we simply point to the recent paper \cite{DDS18} and references therein.

One could weaken the conjecture slightly and guess that $R_p(M)$ is finite for all $M$, given $p \in (1,2)$.
In this article we show that this can only hold if the bound is uniform among all manifolds of a fixed dimension.
This observation follows from the following dichotomy.

\begin{Theorem} \label{thm:main}
  Fix $d \in \N$ and $p \in (1,\infty)$.
  Then the following dichotomy holds: either
  \begin{itemize}
  \item there exists a constant $C_{p,d} < \infty$ such that $R_p(M) \leq C_{p,d}$ for all complete $d$-dimensional Riemannian manifolds $M$, or
  \item there exists a complete $(d+1)$-dimensional Riemannian manifold $M$ such that $R_p(M) = \infty$.
  \end{itemize}
\end{Theorem}
This follows from the following proposition, which we prove by an explicit construction.
\begin{Proposition}\label{prop:construction}
  Fix $d \geq 1$, and let $(M_n)_{n \in \N}$ be a sequence of complete $d$-dimensional Riemannian manifolds.
  Then there exists a complete Riemannian manifold $M$ of dimension $d+1$ such that for all $p \in (1,\infty)$,
  \begin{equation*}
    R_p(M) \geq \sup_{n \in \N}R_p(M_n).
  \end{equation*}
\end{Proposition}

The main implication of Theorem \ref{thm:main} is as follows: to construct a manifold $M$ for which $R_p(M) = \infty$ for some $p \in (1,2)$, it suffices to construct a sequence $(M_n)_{n \in \N}$ of manifolds of equal dimension such that $R_p(M_n) \to \infty$ as $n \to \infty$.
Thus one is led to consider lower bounds for $L^p$-norms of Riesz transforms.
These seem not to have been considered in the literature, excluding of course the well-known computation of the $L^p$-norm of the Hilbert transform (the Riesz transform on $\R$) \cite{sP72}.
We hope that our contribution will provoke further interest in such lower bounds.

\section{Preliminary lemmas}

We begin with some basic lemmas.
The first says that the range of the Laplace-Beltrami operator is dense in $L^p$, and the second relates the Riesz transform on a manifold $M$ with that on the $M$-cylinder $M \times \R$.
These cylinders play a key role in the proof of our main theorem.

\begin{Lemma}\label{lem:range-lap-dens}
Let $M$ be a complete Riemannian manifold.
Then the set $S := \Delta(C_c^\infty(M))$ is dense in $L^p(M)$ for all $p \in (1,\infty)$ (recalling that we write $L^p(M)$ for the space of $p$-integrable mean zero functions when $M$ has finite volume).
\end{Lemma}

\begin{proof}
  Let $H \in L^{p'}(M)$ be such that $\dual{H}{F} = 0$ for every $F \in S$. Then $\dual{H}{\Delta G} = 0$ for every test function $G$, so $H$ is harmonic.
  By \cite[Theorem 3]{sY76}, it follows that $H$ is constant, and the result follows.
\end{proof}

\begin{Lemma}\label{lem:cyl-R}
  Let $M$ be a complete Riemannian manifold.
  Then
  \begin{equation*}
    R_p(M\times \R) \ge R_p(M).
  \end{equation*}
\end{Lemma}

\begin{proof}
Consider the following modification of the Riesz transform on $M \times \R$:
$$\tilde R := \nabla_M (-\Delta_{M\times\R})^{-\frac12} = \nabla_M(-\Delta_M - \partial_t^2)^{-\frac12}.$$ 
This is just the projection of $R$ onto the first summand of the tangent bundle $T(M \times \R) = TM \oplus T\R$, so we have that
\begin{equation}\label{eqn:tilde-R-est}
  \| \tilde RF \|_{L^p} \le \| RF \|_{L^p}.
\end{equation}
Let $F \in C^\infty_c(M \times \R)$, and for all $\lambda > 0$ consider the function 
\begin{equation*}
  F_\lambda(x,t) := \lambda^{\frac1p} F(x,\lambda t),
\end{equation*}
which satisfies $\norm{F_\lambda}_{L^p(M\times\R)} = \norm{F}_{L^p(M\times \R)}$.
Rescaling the operator $\tilde R$ in the variable $t$, we define
\begin{equation*}
  \tilde R_\lambda := \nabla_M(-\Delta_M - \lambda^2 \partial_t^2)^{-\frac12},\
\end{equation*}
so that 
 \begin{equation}\label{eqn:tilde-R}
   \| \tilde R F_\lambda \|_{L^p} = \| \tilde R_\lambda F \|_{L^p}.
 \end{equation}
Now take $f \in C^\infty_c(M) \cap D((-\Delta_M)^{-\frac12})$  and $\rho \in C^\infty_c(\R)$ such that $\norm{\rho}_{L^p(\R)} = 1$, and consider the function $F(x,t) = f(x)\rho(t)$.
Since $\Delta_M$ and $\partial_t^2$ commute, and the function
\begin{equation*}
  G_\lambda(x,y) = \left(\frac{x}{x+\lambda^2 y}\right)^\frac12
\end{equation*}
is bounded by $1$ for $(x,y) > 0$, and $G_\lambda \to 1$ pointwise as $\lambda\to0$, we have
\begin{equation*}
  \lim_{\lambda\to0}(-\Delta_M - \lambda^2 \partial_t^2)^{-\frac12}F = \lim_{\lambda\to0}G_\lambda(-\Delta_M,-\partial_t^2) (-\Delta_M)^{-\frac12} f \otimes \rho = (-\Delta_M)^{-\frac12} f \otimes \rho
\end{equation*}
in $L^2$, and thus also as distributions.
Therefore $\tilde R_\lambda F \to R f \otimes \rho$ as distributions, and so
\begin{equation*}
  \liminf_{\lambda \to 0} \| \tilde R_\lambda F \|_{L^p(M\times \R)} \ge \norm{R f \otimes \rho}_{L^p(M\times \R)} = \norm{Rf}_{L^p(M)}.
\end{equation*}
Combining this with \eqref{eqn:tilde-R} and \eqref{eqn:tilde-R-est}, and the fact that $C^\infty_c(M) \cap D((-\Delta_M)^{-\frac12})$ is dense in $L^p(M)$,\footnote{This follows from the inclusion $D((-\Delta_M)^{-\frac12}) \supseteq D((-\Delta_M)^{-1})  \supseteq \Delta_M(C_c^\infty(M))$, which is dense by Lemma \ref{lem:range-lap-dens}. See also \cite[Lemma 2.2]{MR3445205}. Again, recall that $L^p(M)$ denotes the corresponding space of mean zero functions when $M$ has finite volume.} yields $R_p(M\times \R) \ge R_p(M)$. 
\end{proof}

\section{Proof of the main theorem}

In this section we carry out the construction that proves Proposition \ref{prop:construction}, which implies Theorem \ref{thm:main}.

Consider a sequence $(M_n)_{n \in \N}$ of complete $d$-dimensional Riemannian manifolds.
We will connect the $M_n$-cylinders $(M_n \times \R)_{n \in \N}$ along a $\T^d$-cylinder $\T^d \times \R$ as follows.\footnote{Of course, one could connect the $M_n$-cylinders to each other directly, without needing the $\T^d$-cylinder. This would work just as well.}
For each $n \in \N$ fix a coordinate chart $U_n \subset M_n \times (-1/2,1/2)$ and a small ball $B_n \subset U_n$.
Similarly, for each $n \in \N$ choose a small coordinate chart $U^\prime_n \subset \T^n \times \R$ such that the charts $(U^\prime_n)_{n \in \N}$ are pairwise disjoint, and a small ball $B^\prime_n \subset U^\prime_n$.
For each $n \in \N$, glue the manifold $(M_n \times \R) \setminus B_n$ to $(\T^n \times \R) \setminus B_n^\prime$ along the boundaries of $B_n$ and $B_n^\prime$; this is possible since both these balls are `Euclidean' balls sitting inside coordinate charts.
This results in a $C^0$-Riemannian manifold $(M,g^\prime)$, which is $C^\infty$ away from the set $\Sigma = \cup_n \partial B_n$ on which we glued the manifolds together.
Mollify the metric to get a $C^\infty$-Riemannian manifold $(M,g)$ such that $g = g^\prime$ away from the $\varepsilon$-neighbourhood of $\Sigma$ for some very small $\varepsilon$.
An artist's impression of this construction, with $M_n = S^1$ for each $n$, is shown in Figure \ref{fig:construction}.

\begin{figure}
  \centering
  \def\svgwidth{\columnwidth}
  \includegraphics[width=\textwidth]{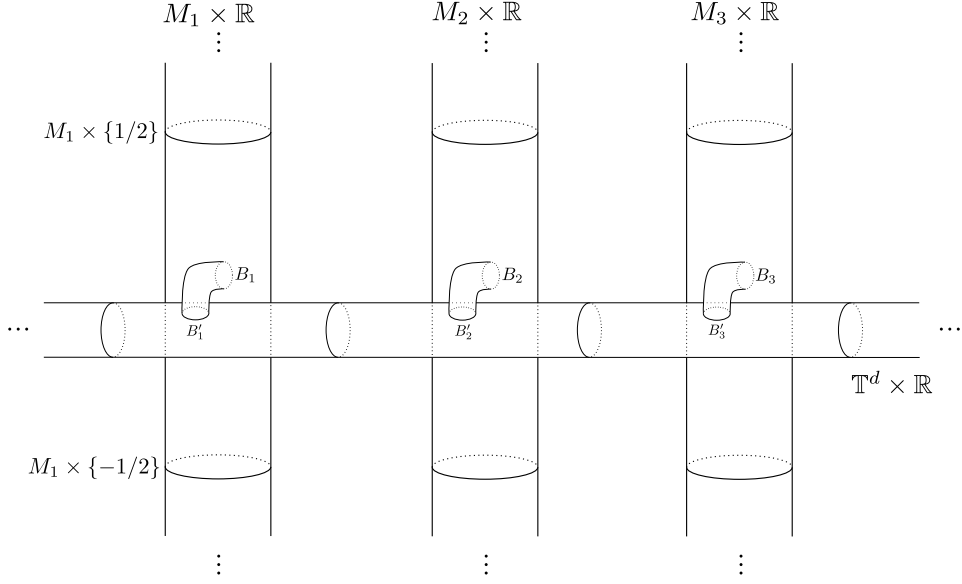}
  \caption{Construction of $M$ from $(M_n)_{n \in \N}$.}
  \label{fig:construction}
\end{figure}

For each $n \in \N$ we have an inclusion map
\begin{equation*}
  i_n \colon M_n \times (1,\infty) \to M
\end{equation*}
which is an isometry.
From here on we fix $n$ and just write $i = i_n$.
Functions on $M$ can be pulled back to $M_n \times (1,\infty)$; the pullback map is denoted $i^*$, so that for $f \colon M \to \R$ the function $i^* f \colon M_n \times (1,\infty) \to \R$ is defined by
\begin{equation*}
  i^*f(x,t) = f(i(x,t)).
\end{equation*}
On the other hand, for $g \colon M_n \times (1,\infty) \to \R$ we can define a pushforward $i_* g \colon M \to \R$ by setting $i_* g(i(x,t)) := g(x,t)$ on $i(M_n \times (1,\infty))$ and extending by zero to the rest of $M$.
For a function $g \colon M_n \times \R \to \R$ and for $s \in \R$ we let $\tau_s g \colon M_n \times \R \to \R$ be the translated function $\tau_s g(x,t) := g(x,t-s)$.
Similarly if $g \colon M_n \times (1,\infty) \to \R$ we can define $\tau_s g \colon M_n \times (1 + s, \infty) \to \R$.
These concepts apply equally well to vector fields in place of functions.

We will need the following lemma, which relates the heat flow on $M_n \times \R$ to the one on $M$.

\begin{Lemma} \label{heat_flow_convergence}
Let $F \colon M_n\times\R\to \R$ be smooth and compactly supported, and fix $\sigma>0$. Then for every $(x,t) \in M_n\times \R$,
\begin{equation*}
  \lim_{s\to +\infty} (e^{\sigma\Delta_M}i_*\tau_s F) (i(x,t+s)) = (e^{\sigma\Delta_{M_n\times\R}}F)(x,t).
\end{equation*}
\end{Lemma}
\begin{proof}
Let $W_{x,t}(\sigma)$ be a Brownian motion on $M_n\times\R$ at time $\sigma$ starting from the point $(x,t)$.
Since the generator $\frac12\Delta_{M_n\times\R}$ satisfies $\frac12i_*\circ\Delta_{M\times \R}|_{i(M_n\times (1,+\infty))} = \frac12\Delta_{M}|_{i(M_n\times (1,+\infty))}$, defining the stopping time 
$$T(x,t):= \inf\left\{s :  W_{x,t}(s) \in M_n \times (-\infty,1) \right\},$$
we have that $i(W_{x,t}(\sigma))$ is a Brownian motion on $M$ for $\sigma < T(x,t)$.
Therefore there exists a Brownian motion $\tilde W_{i(x,t)}(\sigma)$  on $M$ such that $\tilde W(\sigma) = i(W(\sigma))$ for $\sigma<T$; if $\bar W$ is a Brownian motion on $M$, we can take for example 
\begin{equation*}
\tilde W_{i(x,t)}(\sigma) = \left\{ 
\begin{aligned}
&i(W_{x,t}(\sigma)) &\text{if } \sigma < T, \\
&\bar W_{i(W_{x,t}(T))}(\sigma - T) &\text{if } \sigma\ge T.
\end{aligned}\right.
\end{equation*}
We have that
\begin{align*}
&(e^{\sigma\Delta_M}i_*\tau_s F) (i(x,t+s)) \\
&= \E[(i_*\tau_s F)(\tilde W_{i(x,t+s)}(2\sigma))]\\
&= \E[(i_*\tau_s F)(\tilde W_{i(x,t+s)}(2\sigma)) \1_{2\sigma<T}] + \E[(i_*\tau_s F)(\tilde W_{i(x,t+s)}(2\sigma))\1_{2\sigma\ge T}]\\
&= \E[(\tau_s F)(W_{x,t+s}(2\sigma))\1_{2\sigma<T}] +  \E[(i_*\tau_s F)(\tilde W_{i(x,t+s)}(2\sigma))\1_{2\sigma\ge T}]\\
& \begin{multlined}
=\E[(\tau_s F)(W_{x,t+s}(2\sigma))]  \\
- \E[(\tau_s F)(W_{x,t+s}(2\sigma))\1_{2\sigma\ge T}] +  \E[(i_*\tau_s F)(\tilde W_{i(x,t+s)}(2\sigma))\1_{2\sigma\ge T}]
\end{multlined}\\
&\begin{multlined}
= (e^{\sigma\Delta_{M_n\times\R}}\tau_sF)(x,t+s)\\
 - \E[(\tau_s F)(W_{x,t+s}(2\sigma))\1_{2\sigma\ge T}] +  \E[(i_*\tau_s F)(\tilde W_{i(x,t+s)}(2\sigma))\1_{2\sigma\ge T}].
\end{multlined}
\end{align*}
Therefore 
\begin{equation*}
 \left|(e^{\sigma\Delta_M}i_*\tau_s F) (i(x,t+s)) - (e^{\sigma\Delta_{M_n\times\R}}\tau_sF)(x,t+s)\right|
\le 2 \norm{F}_{L^\infty}\P(T(x,t+s)\le2\sigma).
\end{equation*}
Since $\Delta_{M_n\times \R}$ is translation invariant in the $\R$ coordinate, we have that 
\begin{align*}
\P(T(x,t+s)\le 2\sigma) &\le \P\big(\{W_{x,t+s}(\sigma') \in M_n\times (-\infty,1) \text{ for some }\sigma'\le2\sigma+1\}\big) \\
&=\P\big(\{W_{x,t}(\sigma') \in M_n\times (-\infty,1-s) \text{ for some }\sigma'\le2\sigma+1\}\big)
\end{align*}
and by continuity of $W_{x,t}(\cdot)$, this tends to $0$ as $s \to \infty$.
Thus we find that
\begin{equation*}
  \lim_{s\to +\infty} \bigg( (e^{\sigma\Delta_M}i_*\tau_s F) (i(x,t+s)) - (e^{\sigma\Delta_{M_n\times\R}}\tau_sF)(x,t+s) \bigg) = 0.
\end{equation*}
The conclusion follows from translation invariance of $\Delta_{M_n\times\R}$ in $\R$.
\end{proof}

We return to the proof of Proposition \ref{prop:construction}.
Fix $\epsilon > 0$, and choose $F = \Delta_{M_n \times \R} H$ for some $H \in C^\infty_c(M_n \times \R)$ with $\norm{F}_{L^p} = 1$ such that 
\begin{equation*}
  \norm{R_{M_n\times\R} F}_{L^p} \ge (R_p(M_n)-\epsilon) \wedge \epsilon^{-1}.
\end{equation*}
Such a function exists by Lemmas \ref{lem:range-lap-dens} and \ref{lem:cyl-R}.
We claim that
\begin{equation}\label{eqn:dist-limit}
  \lim_{s \to +\infty} \tau_{-s} i^*R_{M}(i_* \tau_s F) = R_{M_n\times \R} F
\end{equation}
as distributions.
Assuming \eqref{eqn:dist-limit} for the moment, we have
\begin{align*}
  \limsup_{s \to \infty} \norm{R_{M}(i_* \tau_s F)}_{L^p(M)} &\ge \limsup_{s \to \infty} \norm{i^*R_{M}(i_* \tau_s F)}_{L^p(M_n\times\R)}\\
                                                             &= \limsup_{s \to \infty} \norm{\tau_{-s}i^*R_{M}(i_* \tau_s F)}_{L^p(M_n\times\R)}\\
&\ge \norm{R_{M_n\times \R} F}_{L^p(M_n\times\R)} \ge R_p(M_n)-\epsilon,
\end{align*}
while for all $s \in \R$
$$\norm{i_* \tau_s F}_{L^p(M)} \leq \norm{\tau_s F}_{L^p(M_n\times\R)} = \norm{F}_{L^p(M_n\times\R)}\le 1. $$
The result follows, so it remains to prove \eqref{eqn:dist-limit}.

For $s$ sufficiently large, we have that 
$$i_*\tau_sF = i_*\tau_s(\Delta_{M_n \times \R} H)  = i_*(\Delta_{M_n \times \R}\tau_s H) = \Delta_M i_* \tau_s H,$$
therefore $i_*\tau_sF \in D(\Delta_M^{-1}) \subseteq D((-\Delta_M)^{-\frac12})$, and hence 
$$R(i_*\tau_sF) = \nabla \left((-\Delta)^{-\frac12}_Mi_*\tau_sF\right)$$
as a distribution. To test the distributional convergence, let $X$ be a smooth compactly supported vector field in
$M_n \times \R$.
For large $s$ we have that 
\begin{align*}
  \dual{\tau_{-s} i^*R_{M}(i_* \tau_s F)}{X} &= \dual{R_{M}(i_* \tau_s F)}{i_*\tau_sX} \\
                                             &= \dual{(-\Delta)^{-\frac12}_Mi_*\tau_sF}{\div(i_*\tau_sX)} \\
                                             &= \dual{(-\Delta)^{-\frac12}_Mi_*\tau_sF}{i_*\tau_s\div(X)}.
\end{align*}
Therefore it is enough to show that for every $G \in C^\infty_c(M_n\times \R)$, 
\begin{equation}\label{limit}
\lim_{s\to \infty}\dual{(-\Delta)^{-\frac12}_Mi_*\tau_sF}{i_*\tau_sG} = \dual{(-\Delta)^{-\frac12}_{M_n\times\R} F}{G}.
\end{equation}
By the well-known formula
$$(-\Delta)^{-\frac12} = \pi^{-\frac12} \int_0^{+\infty} \sigma^{-\frac12}e^{\sigma\Delta} \, d\sigma,$$
\eqref{limit} is equivalent to showing that 
\begin{equation}\label{eqn:limit-integrals}
  \lim_{s\to \infty}\int_0^{+\infty} \sigma^{-\frac12} \dual{e^{\sigma\Delta_M}i_*\tau_sF}{i_*\tau_sG} \, d\sigma
  = \int_0^{+\infty} \sigma^{-\frac12} \dual{e^{\sigma\Delta_{M_n\times\R}} F}{G} \, d\sigma.
\end{equation}
Note that  
\begin{equation*}
  \left|\sigma^{-\frac12} \dual{e^{\sigma\Delta_M}i_*\tau_sF}{i_*\tau_sG}\right| \le \sigma^{-\frac12} 
  \norm{i_*\tau_sF}_{L^2}\norm{i_*\tau_sG}_{L^2}\le \sigma^{-\frac12} \norm{F}_{L^2}\norm{G}_{L^2}\
\end{equation*}
and
\begin{equation*}
\resizebox{\textwidth}{!}{$\displaystyle
  \left|\sigma^{-\frac12} \dual{e^{\sigma\Delta_M}i_*\tau_sF}{i_*\tau_sG}\right| = \left|\sigma^{-\frac32} \dual{e^{\sigma\Delta_M}\sigma\Delta_Mi_*\tau_sH}{i_*\tau_sG}\right| \lesssim \sigma^{-\frac32} \norm{H}_{L^2}\norm{G}_{L^2}.$}
\end{equation*}
Since the function $\min(\sigma^{-\frac12},\sigma^{-\frac32})$ is integrable, by dominated convergence \eqref{eqn:limit-integrals} will be proved if we show
\begin{equation}\label{eqn:limit-integrand}
  \lim_{s\to \infty} \dual{e^{\sigma\Delta_M}i_*\tau_sF}{i_*\tau_sG} =  \dual{e^{\sigma\Delta_{M_n\times\R}} F}{G}
\end{equation}
for every $\sigma > 0$.
We show \eqref{eqn:limit-integrand} by writing
\begin{align*}
\lim_{s\to\infty}\dual{e^{\sigma\Delta_M}i_*\tau_sF}{i_*\tau_sG} & = \lim_{s\to\infty}\dual{\tau_{-s}i^*e^{\sigma\Delta_M}i_*\tau_sF}{G}\\
& = \lim_{s\to\infty} \int_{1-s}^{+\infty} \int_{M_n} (e^{\sigma\Delta_M}i_*\tau_s F) (i(x,t+s))G(x,t) \, dx \, dt \\
& = \int_\R \int_{M_n} (e^{\sigma\Delta_{M_n\times\R}} F)(x,t) G(x,t) \, dx \, dt\\
& = \dual{e^{\sigma\Delta_{M_n\times\R}} F}{G},
\end{align*}
using Lemma \ref{heat_flow_convergence} and dominated convergence (by $\norm{F}_{L^\infty} |G(x,t)|$).
This completes the proof of Proposition \ref{prop:construction}, and hence establishes Theorem \ref{thm:main}.

\footnotesize

\bibliographystyle{plain}

\bibliography{riesz_failure}

\end{document}